\documentclass{amsart}
\usepackage{amssymb}
\usepackage{amsmath, amscd}
\usepackage{amsthm}
\usepackage{mathrsfs}
\usepackage{perpage}
\usepackage[all,cmtip]{xy}
\usepackage{setspace}
\usepackage{amsbsy}
\usepackage[foot]{amsaddr}
\usepackage{url}

\theoremstyle{plain}
\newtheorem{thm}{Theorem}[section]
\newtheorem{prop}[thm]{Proposition}
\newtheorem{lem}[thm]{Lemma}
\newtheorem{cor}[thm]{Corollary}
\newtheorem{remark}[thm]{Remark}

\theoremstyle{definition}

\newtheorem{dfn}[thm]{Definition}

\theoremstyle{remark}

\newtheorem{rmk}[thm]{Remark}

\catcode`\Ž=\active\def Ž{\'e}
\catcode`\=\active\def {\`e}
\catcode`\ˆ=\active\def ˆ{\`a}
\catcode`\‰=\active\def ‰{\^a}
\catcode`\Ë=\active\def Ë{\`A}
\catcode`\Š=\active\def Š{\"a}
\catcode`\=\active\def {\`u}
\catcode`\=\active\def {\^{e}}
\catcode`\"=\active\def "{\^{i}}
\catcode`\•=\active\def •{\"i}
\catcode`\™=\active\def ™{\^{o}}
\catcode`\š=\active\def š{\"{o}}
\catcode`\ž=\active\def ž{\^{u}}
\catcode`\Ÿ=\active\def Ÿ{\"{u}}
\catcode`\†=\active\def †{\"{U}}
\catcode`\=\active\def {\c{c}}
\catcode`\'=\active\def '{\c{C}}
\catcode`\é=\active\def é{\c{C}}

\newcommand{\Ha}{\ \! ^\mu \! H}

\newcommand{\shgx}{\Sh(\mathbf G, \mathbf X)}

\newcommand{\ul}{^{(p)}}
\newcommand{\uls}{^{(p^2)}}

\newcommand{\gofa}{\mathbf G(\mathbf A)}

\newcommand{\qlbar}{\overline{\mathbf Q}_{p}}
\newcommand{\glnqlbar}{\GL(n, \qlbar)}
\newcommand{\galf}{ \Gal (\overline{F}/F)}
\newcommand{\rlipi}{R_{p, \iota}(\pi)}
\newcommand{\rec}{\rm{rec}}
\newcommand{\sesi}{^{\rm{ss}}}
\newcommand{\Th}{{\rm Th.}}
\newcommand{\Ths}{{\rm Ths.}}
\newcommand{\Cor}{{\rm Cor.}}
\newcommand{\Defin}{{\rm Def.}}
\newcommand{\diag}{{\operatorname{diag}}}

\newcommand{\Grad}{{\operatorname{Gr }}}
\newcommand{\Hom}{{\operatorname{Hom}}}

\newcommand{\lcm}{{\operatorname{lcm}}}

\newcommand{\ord}{{\operatorname{ord }}}

\newcommand{\rank}{\mbox{rank }}

\newcommand{\Spec}{{\operatorname{Spec }}}

\newcommand{\Sh}{{\operatorname{Sh }}}
\newcommand{\GL}{{\operatorname{GL}}}

\newcommand{\Gal}{{\operatorname{Gal}}}
\newcommand{\gerp}{{\mathfrak{p}}}

\newcommand{\calA}{{\mathcal{A}}}

\newcommand{\calO}{{\mathcal{O}}}

\newcommand{\card}{\mbox{ card}}

\def\2vector#1#2{\left( \begin{smallmatrix} #1 \\ #2 \end{smallmatrix}
\right)}

\def\deb{ \begin{equation} }
\def\fin{ \end{equation} }

\usepackage[usenames]{color}
\definecolor{Indigo}{rgb}{0.2,0.1,0.7}
\definecolor{Violet}{rgb}{0.5,0.1,0.7}
\definecolor{White}{rgb}{1,1,1}
\definecolor{Green}{rgb}{0.1,0.9,0.2}

\newcommand{\loccit}{{\em loc. cit.}}

\newcommand{\crys}{{\rm crys}}
\newcommand{\fil}{{\rm Fil}}
\newcommand{\np}{{\rm NP}}

\begin{document}

\title{The $\mu$-ordinary Hasse invariant of unitary Shimura varieties}
\date{\today}
\author{Wushi Goldring, Marc-Hubert Nicole}

\begin{abstract} We construct a generalization of the Hasse invariant for any Shimura variety of PEL type $A$ over a prime of good reduction, whose non-vanishing locus is the open and dense $\mu$-ordinary locus.

\end{abstract}

\address{W. G. Institut f\"ur Mathematik
Mathematisch-naturwissenschaftliche Fakult\"at
Universit\"at Z\"urich, Winterthurerstrasse 190, CH-8057 Z\"urich, SWITZERLAND}
\email{wushijig@gmail.com}
\address{M.-H. N. Aix Marseille UniversitŽ, CNRS, Centrale Marseille, I2M, UMR 7373, 13453 Marseille; Mailing Address: UniversitŽ d'Aix-Marseille, campus de Luminy, case 907, Institut mathŽmatique de Marseille (I2M), 13288 Marseille cedex 9, FRANCE}
\email{marc-hubert.nicole@univ-amu.fr}
\thanks{M.-H. N. thanks the Max Planck Institut fŸr Mathematik (MPIM, Bonn) for a year-long membership in 2011.}

\subjclass[2010]{Primary 14G35; Secondary 11F33, 11F55.}
\keywords{Hasse invariant, Shimura varieties, $\mu$-ordinary locus, Galois representations}
\maketitle

\numberwithin{equation}{section}

\section{Introduction}

Let $p$ be a prime number and let $sh$ be a special fiber modulo $p$ of a Shimura variety of PEL type at a neat level which is hyperspecial at $p$. The classical Hasse invariant $H$ is, roughly speaking, an automorphic form mod $p$ of weight $p-1$. The classical Hasse invariant satisfies the following four properties:
\begin{enumerate}
\item[(Ha1)]
The non-vanishing locus of $H$ is the ordinary locus of $sh$, namely the locus of points where the underlying abelian variety is ordinary.
\item[(Ha2)]
The construction of $H$ is compatible with varying the prime-to-$p$ level.
\item[(Ha3)]
A power of $H$ extends to the minimal compactification of $sh$.
\item[(Ha4)]
A power of $H$ lifts to characteristic zero.
\end{enumerate}

The Hasse invariant is the main tool to construct congruences modulo powers of $p$, both in the realms of automorphic forms and of Galois representations. However, when $\frak p$ is a prime of the reflex field $E$ of the Shimura variety for which the $\frak p$-adic completion $E_{\frak p}$ is strictly larger than $\mathbf Q_p$, the ordinary locus is empty and the Hasse invariant is identically zero.

To fix this, we construct a generalized Hasse invariant satisfying properties (Ha2)--(Ha4) and a ``$\mu$-ordinary" analogue of (Ha1) for any Shimura variety ${\rm Sh}(\mathbf G, \mathbf X)$ of PEL-type such that $\mathbf G$ is a group of unitary similitudes. The non-vanishing locus of our generalized Hasse invariant is the $\mu$-ordinary locus, which, as Moonen has shown \cite[\Ths 1.3.7, 3.2.7]{MoonenSerreTate}, is simultaneously the largest stratum of the Newton and of the Ekedahl-Oort stratifications. As an application, we use our new Hasse invariant to generalize the main result of \cite{WushiGalrepshldsI}, which concerns attaching Galois representations to automorphic representations whose archimedean component is a holomorphic limit of discrete series.

The main idea in this paper is to use the action of Frobenius $\mathbf F$ on the crystalline cohomology of abelian varieties. The use of this cohomology theory allows us to divide by $p$ i.e., to make sense of the operator ``$\wedge^i\mathbf F/p^j$'' for well-chosen positive integers $i$ and $j$, see below. In the main body of the paper, we pursue the Newton point of view and apply the Newton-Hodge decomposition of Katz, a convenient tool in this context. In the first appendix, we illustrate how we can retrieve most of our results purely from the Ekedahl-Oort point of view. In the second appendix, we show how we can avoid the use of crystalline cohomology when the totally real field $F^+$ is equal to $\mathbf Q$ or, equivalently, that $\mathbf G(\mathbf R)$ is isomorphic to the unitary group $\mathbf{GU}(a,b)$ for some $a,b \in \mathbf N_{>0} $.

We note that this article is the result of merging our two arXiv postings \cite{GoldringNicole} and \cite{GoldringNicole2}.    We also remark that, a little over one year after we posted \cite{GoldringNicole2} on arXiv, Koskivirta and Wedhorn posted a preprint in which they construct generalized Hasse invariants for Shimura varieties of Hodge type, see \cite{KoWe}.

\subsection{Main Results} \label{sec main result}
Throughout this paper, fix an isomorphism $\iota: \qlbar \stackrel{\sim}{\longrightarrow }\mathbf C$.

Suppose ${\mathcal U}=(B, V,* <,>, \tilde{h})$ is a Kottwitz datum with associated Shimura variety $\shgx$, such that the center of the simple $\mathbf Q$-algebra $B$ is a totally imaginary quadratic field extension $F$  of a totally real field $F^+$ \cite[3.1]{WushiGalrepshldsI}. Let $d$ be the degree of $F^+$ over $\mathbf Q$. Suppose $p$ is a prime of good reduction for $\mathcal U$ (see \loccit \ \S3.3) and $\mathcal K \ul \subset \mathbf G(\mathbf A_f^{p})$ is a neat, open compact subgroup.

Let $E=E(\mathbf G, \mathbf X)$ be the reflex field of $\shgx$.
Let $Sh:= Sh_{\mathcal K \ul}$ be the Kottwitz integral model of $Sh(\mathbf G, \mathbf X)$ at level $\mathcal K\ul$ over $\mathbf Z_{(p)}\otimes \mathcal O_E$. Let $\gerp$ be a prime of $E$ above $p$. Denote by $sh:= sh_{\mathcal K\ul, \gerp}$ the special fiber of $Sh_{\mathcal K\ul}$ at $\gerp$. Let $\omega$ be the Hodge line bundle of $sh$ as defined in \S\ref{sec f crystals hodge filtration}.

\begin{thm} There exists an explicit positive integer $m \in \mathbf Z_{\geq 1}$ and a section \begin{equation} ^\mu \! H \in H^0(sh, \omega^{m}) \label{eq mu Hasse} \end{equation} satisfying the following four properties:
\begin{enumerate}
\item[($\mu$-Ha1)] The non-vanishing locus of $^\mu \! H$ is the $\mu$-ordinary locus of $sh$, as defined in \cite{RapRic} and \cite{Wed}.

\item[($\mu$-Ha2)]
The construction of $^\mu \! H$ is compatible with varying the level $\mathcal K\ul$.
\item[($\mu$-Ha3)]
The section $^\mu \! H$ extends to the minimal compactification.
\item[($\mu$-Ha4)]
A power of $^\mu \! H$ lifts to characteristic zero.
\end{enumerate}

\label{th mu ordinary hasse invariant}\end{thm}
\noindent
We call $^\mu \! H$ the $\mu$-ordinary Hasse invariant.

\begin{rmk} The exponent $m$ in \Th~\ref{th mu ordinary hasse invariant} is explicitly defined in \Defin~\ref{def exponent m}, in terms of the action of Frobenius on the embeddings of $F$. In case $p$ remains prime in $F$, the formula one finds there simplifies to $m=p^{2d}-1$. \end{rmk}

By ampleness of the Hodge line bundle $\omega$ on the minimal compactification (cf. \cite[Th. 7.2.4.1, no.2]{Lan}), we deduce the following corollary:

\begin{cor} \label{cor affine}
The $\mu$-ordinary locus $sh_{}^{min, \mu-{\rm ord}}$ in the minimal compactification $sh^{{\rm min}}$ is affine.
\end{cor}

\subsection{Application to Galois representations}

We also obtain an application to the construction of automorphic Galois representations which generalizes \cite[\Th 1.2.1]{WushiGalrepshldsI}. To state the result we need some notation.

Suppose $\pi$ is a cuspidal automorphic representation of $\gofa$ with $v$-adic component $\pi_v$ for every place $v$. Given a prime $p$, let ${\mathcal P}^{(p)}$ be the set of primes $v$ different from $p$ such that $\pi_v$ is unramified and $\mathbf G$ is unramified at $v$. Let $\frak P ^{(p)}$ be the set of primes of $F$ that are split over $F^+$ and lie over some $v \in {\mathcal P}^{(p)}$.

Assume $w \in \frak P^{(p)}$. One has a decomposition $\mathbf G(\mathbf Q_v) \cong \GL(n, F_{w}) \times G_{v,{\rm rest}}$,  for some group $G_{v,{\rm rest}}$, where $n$ is given by $n^2=\dim_F{\rm End}_B V$. Write $\pi_v \cong \pi_{w}\otimes \pi_{v,{\rm rest}}$, with $\pi_{w}$ a representation of $\GL(n, F_{w})$  and $\pi_{v,{\rm rest}}$ a representation of $G_{v, {\rm rest}}$.

\begin{thm}
Suppose $\pi$ is a cuspidal automorphic representation of $\mathbf G(\mathbf A)$ whose archimedean component $\pi_{\infty}$ is an $\mathbf X$-holomorphic limit of discrete series representation of $\mathbf G( \mathbf R)$ {\rm (}see \cite[\S2.3]{WushiGalrepshldsI}{\rm )}. Assume $p$ is a prime of good reduction for $\mathcal U$.
Then there exists a unique semisimple Galois representation
\begin{equation}\rlipi:\galf \longrightarrow  \glnqlbar\label{eq gal rep lds} \end{equation}
satisfying the following two conditions:
\begin{enumerate} \item[Gal1.]
If $v \in \mathcal P^{(p)}$ and $w$ is a prime of $F$ dividing $v$ then $\rlipi$ is unramified at $w$. In particular $\rlipi$ is unramified at all but finitely many places.

\item[Gal2.]  If $w \in \frak P^{(p)}$ then there is an isomorphism of Weil-Deligne representations
 \begin{equation} (\rlipi | _{W_{F_{w}}})^{\sesi}\cong \iota^{-1}{\rec}\big( \pi_{w} \otimes |\cdot|_{w}^{\frac{1-n}{2}}\big), \end{equation} where $W_{F_{w}}$ is the Weil group of $F_{w}$, the superscript $\sesi$ denotes semi-simplification and $\rec$ is the local Langlands correspondence, normalized as in {\rm \cite{HT}}.
 \end{enumerate}
\label{cor main unitary 1} \end{thm}

\begin{rmk} The argument given in \S6 of \cite{WushiGalrepshldsI} carries over almost verbatim (see \S\ref{sec correction} for a minor correction) and shows that our main result \Th~\ref{th mu ordinary hasse invariant} implies our application \Th~\ref{cor main unitary 1}. \end{rmk}

\section{Preliminaries on $F$-crystals and the Hodge filtration} \label{sec f crystals hodge filtration}

Let $E \subset E'\subset \mathbf C$, where $E'$ is a finite extension of $E$ such that  $B$ is split over $E'$ and
for every embedding $\tau: F\hookrightarrow \mathbf C$, one has $\tau(F) \subset E'$. Denote by $\gerp$ a prime of $E$ over $p$, and by $\gerp'$ a prime of $E'$ over $\gerp$. Pick $\kappa$ to be the smallest finite field containing the residue fields $\mathcal O_{E'}/\gerp'$, for all $\gerp'$ over $\gerp$. Via $\iota: \qlbar \stackrel{\sim}{\longrightarrow } \mathbf C$, there is a bijection $\tau \mapsto \iota^{-1} \circ \tau$ between the set of complex embeddings $\tau:F \hookrightarrow \mathbf C$ and the set of $p$-adic embeddings $\iota^{-1} \circ \tau:F \hookrightarrow \qlbar$ and we denote either type of embedding simply by $\tau$. After fixing an embedding  $W(\kappa) \hookrightarrow \mathbf C$, there is further a bijection with the set of embeddings of $\mathcal O_F$ into $W(\kappa)$, and also with the set of homomorphisms to $\kappa$, noted $\Hom(\mathcal O_F, \kappa)$. The absolute Frobenius, noted $\sigma$, acts via composition on $\Hom(\calO_F, \kappa)$.

Let $\mathcal T$ be the set of complex embeddings of $F$. Let $r$ be the rank of $B$ over $F$. From here onwards, fix the prime $\gerp$ in $\calO_E$. Let $S$ be a smooth $\Spec (\mathcal O_E/\gerp)$-scheme and $\pi:A\rightarrow S$ a $\mathcal U\ul$-enriched abelian scheme \cite[\S3.4]{WushiGalrepshldsI}.
Let $\omega = \bigwedge^{\rm top} \pi_* \Omega^1_{A/S}$ be the Hodge bundle i.e., the determinant of the pushforward of the sheaf of relative differentials on $A$.
After extending scalars to $\kappa$, the Hodge bundle decomposes according to the embeddings $\tau \in \mathcal T$ and the standard idempotents in $M_r(\kappa)$:
\[ \omega = \bigotimes_{\tau \in \mathcal T} \omega_{\tau}^{\otimes r} \]
The Dieudonn\'e crystal $H^1_{\crys}(A)$ also decomposes accordingly:
\[ H^1_{\crys}(A) = \bigoplus_{\tau\in \mathcal T} H^1_{\crys}(A)_{\tau}^{\oplus r}. \]

\noindent Similarly for de Rham cohomology, one has:
\[ H^1_{dR}(A) = \bigoplus_{\tau\in\mathcal T} H^1_{dR}(A)_{\tau}^{\oplus r}. \]

Put $H^{d}_{\crys}(A)_{\tau_i}=\bigwedge^{d}H^1_{\crys}( A)_{\tau_i}$ and $H^{d}_{dR}(A)_{\tau_i}=\bigwedge^{d}H^1_{dR}(A)_{\tau_i}$

\noindent Let $\fil^\bullet$ denote the Hodge filtration on the de Rham cohomology. Put $\fil^1_{\tau}= \fil^1 H^1_{dR}( A) \cap H^1_{dR}(A)_{{\tau}}$. Then  $(\rank \fil^1_{{\tau}}, \rank \fil^1_{{\overline{\tau}}})$ is the signature corresponding to the conjugate pair of embeddings $(\tau, \overline{\tau})$.

 Given $\tau \in \mathcal T$, define $\frak o_{\tau}$ to be the orbit of $\tau$ under the action of the absolute Frobenius $\sigma$. Let $e_{\tau}$ denote the cardinality of the orbit $\frak o_{\tau}$. Write $\mathfrak o_{\tau}=\{\tau_1, \ldots , \tau_{e_{\tau}}\}$ in such a way that $\rank \fil^1_{\tau_1} \geq \cdots \geq \rank \fil^1_{\tau_{e_{\tau}}}$. The rank of $\ H^1_{dR}(A)_{\tau}$ is independent of $\tau$; we call it $n$. Define the multiplication type $\frak f:\frak o_{\tau}\rightarrow \{0,1,\ldots, n\}$ associated to $\frak o_{\tau}$ by $\frak f(\tau_i)=\rank \fil^1_{\tau_i}$. To the pair $(n, \frak f)$ depending on $\frak o_{\tau}$, Moonen \cite[1.2.5]{MoonenSerreTate} associates a polygon $\ord_{\frak o_{\tau}}(n, \frak f)$ that we call the $\mu$-ordinary polygon associated to $\frak o_{\tau}$. Recall that the slopes $a_j$, $1 \leq j \leq n$, of $\ord_{\frak o_{\tau}}(n, \frak f)$ are defined by \begin{equation} a_j := \card(\{\tau' \in \frak o_{\tau}|\frak f(\tau') >n-j\}). \label{eq def slope ord}\end{equation}

Now suppose $S=\Spec\ k$, where $k$ is an algebraically closed field, so that $A$ represents a geometric point of $sh$. Put $M=H^1_{crys}(A)$ (resp. $M_{\tau}=H^1_{crys}(A)_{\tau}$). Define the Hodge (resp. Newton) polygon of $M_{\tau}$ to be the Hodge (resp. Newton) polygon of $(M_{\tau},\mathbf F^{e_{\tau}})$. Note that in general the Newton polygon of $M_{\tau}$ does not depend on ${\tau}$ but the Hodge polygon does.

\begin{lem} Let $\mathcal T=\coprod \mathfrak o_{\tau}$ be the orbit decomposition of $\mathcal T$ according to the action of Frobenius. Let $M=H^1_{\crys}(A)$. Then the Newton polygon of $(M, \mathbf F)$ is the Newton polygon $\np(sh^{\mu-\ord})$ of the $\mu$-ordinary locus (i.\ e.,\ $A$ is $\mu$-ordinary) if and only if for all $\tau \in \mathcal T$ the Newton polygon of $M_{\tau}$ is the $\mu$-ordinary polygon $\ord_{\frak o_{\tau}}(n, \frak f)$.  \label{lemma orbits mu ordinary} \end{lem}
\begin{proof} See \cite[2.2.1]{Wed}. \end{proof}

\begin{lem} Suppose the Newton polygon of $M_{\tau}$ is the $\mu$-ordinary polygon $\ord_{\frak o_{\tau}}(n, \frak f)$. Then the Hodge polygon of $M_{\tau}$ also coincides with $\ord_{\frak o_{\tau}}(n, \frak f)$. \label{lemma hodge newton polygon} \end{lem}
In particular, under the assumption of being $\mu$-ordinary, the Hodge polygon of $M_{\tau}$ depends only on the orbit ${\frak o_{\tau}}$.
\begin{proof} This follows from the proof of \cite[1.3.7]{MoonenSerreTate}.
\end{proof}

Suppose $A$ is $\mu$-ordinary. Combining Lemmas~\ref{lemma orbits mu ordinary} and~\ref{lemma hodge newton polygon}, both the Newton polygon and the Hodge polygon of $M_{\tau}$ are equal to $\ord_{\frak o_{\tau}}(n,\frak f)$. By the Hodge-Newton decomposition \cite[1.6.1]{KatzSlopeFiltration}, one can write\begin{equation} M_{\tau}=\bigoplus_{j=0}^{e_{\tau}} M_{\tau}^{[j]}, \end{equation} where $M_{\tau}^{[j]}$ is an isoclinic subcrystal of slope $j$.
\begin{thm} One has \begin{equation} \fil^1_{\tau_i}=\bigoplus_{j\geq i}\overline{M_{\tau_i}^{[j]}} \label{eq fil1 slope decomp} .\end{equation}\label{th fil1 slope decomp} \end{thm}
\begin{proof} By the explicit description of $\ord_{\frak o_{\tau}}(n,\frak f)$, we see that the two sides of~\eqref{eq fil1 slope decomp} have the same dimension. Hence it suffices to show the inclusion \begin{equation} \fil^1_{\tau_i} \supset \bigoplus_{j\geq i}\overline{M_{\tau_i}^{[j]}}. \end{equation}
To this effect, our main tool will be Mazur's theorem which we recall now.
\begin{thm}[Mazur] Let $A$ be an abelian variety over an algebraically closed field $k$ of characteristic $p$. Denote by $\overline{\ \square \ }$ the reduction modulo $p$ of $\square$. Then for all $j, m \in \mathbf Z_{\geq 0}$, one has \begin{equation} \fil^j H^m_{dR}(A)=\overline{\mathbf F^{-1} (p^jH^m_{\crys}(A))}, \label{eq frobenius determines the hodge filtration} \end{equation} where $\mathbf F$ is the canonical lifting of Frobenius on crystalline cohomology.\label{th frobenius determines the hodge filtration} \end{thm}
\begin{proof} Let $A$ be an abelian variety over an algebraically closed field of characteristic 0. Since the Hodge-de Rham spectral sequence of $A$ degenerates at $E_1$ and since the crystalline cohomology of $A$ torsion-free, the theorem is a special case of \cite[8.26]{BerOgu}. \end{proof}

\Th~\ref{th fil1 slope decomp} will now be proved as follows: Lemma~\ref{lemma isogeny preserves slopes}-\Cor~\ref{cor h1 crys 2} are of a preparatory nature. The crux of the proof of \Th~\ref{th fil1 slope decomp} is contained in Lemmas~\ref{lemma base fil1 slopes} and~\ref{lemma induction fil1 slopes}.

\begin{lem} \label{isogenies-respect-decomposition} Suppose, for $i \in \{1,2\}$, that $(M_i,\mathbf F_i)$ is a $W(k)$-module which is an ordinary $\mathbf F_i$-crystal i.\ e., the Hodge and Newton polygons of $(M_i,\mathbf F_i)$ coincide. Let $\varphi:(M_1, \mathbf F_1) \rightarrow (M_2, \mathbf F_2)$ be an isogeny, so in particular the Newton polygon of $(M_1,\mathbf F_1)$ is the same as that of $(M_2,\mathbf F_2)$. Let $0\leq \lambda_1< \cdots <\lambda_s$ be the slopes of $(M_i,\mathbf F_i)$ with multiplicities $m_1, \ldots, m_s$. Let \begin{equation} M_i=\bigoplus_{j=1}^sM_{i,j} \end{equation} be the Newton-Hodge decomposition {\rm \cite[1.6.1]{KatzSlopeFiltration}} applied to $(M_i,\mathbf F_i)$ so that $(M_{i,j},\mathbf F_i)$ is an isoclinic subcrystal of rank $m_j$ and slope $\lambda_j$. Then $\varphi (M_{1,j}) \subset M_{2,j}$ \label{lemma isogeny preserves slopes} \end{lem}
\begin{proof} Since $\varphi$ is an isogeny and $M_{1,j}$ is a $\mathbf F_1$-subcrystal of $M_1$, the image $\varphi(M_{1,j})$ is an $\mathbf F_2$-subcrystal of $M_2$. Since the Newton polygon is invariant under isogeny, $(\varphi(M_{1,j}),\mathbf F_2)$ is isoclinic of slope $\lambda_j$ with multiplicity $m_j$. Let $M'$ denote the $\mathbf F_2$ subcrystal of $M_2$ generated by $\varphi(M_{1,j})$ and $M_{2,j}$. Then $M'$ is isoclinic of slope $\lambda_j$, so the rank of $M'$ is $m_j$. Since $M_2/M_{2,j}$ is free, we conclude that $M'=M_{2,j}$. Therefore $\varphi(M_{1,j}) \subset M_{2,j}$.
\end{proof}

\begin{remark}
{\rm Lemma \ref{isogenies-respect-decomposition} also follows more generally from the fact that homomorphisms of $F$-crystals respect the slope decomposition, see \cite[Property e), p. 81]{Dem}.}

\end{remark}

\begin{lem} Let $M=H^1_{\crys}(A)$, where $A$ is an abelian variety. Suppose $k \in \mathbf Z_{\geq 2}$, $ x \in M$ and $\mathbf F(x) \in p^kM$. Then $x \in p^{k-1}M$. \label{lemma h1 crys} \end{lem}
\begin{proof} Since $\mathbf F(x) \in p^kM$, Mazur's theorem entails that $\bar{x} \in \fil^k\bar{M}$. But $k \geq 2$, so $\fil^k\bar{M}=\{0\}$. Hence $\bar{x}=0$, so $x \in pM$. If $k=2$ we are done. So assume $k>2$ and write $x=py$, for some $y \in M$. Then $\mathbf F(x)=\mathbf F(py)=p\mathbf F(y)$ and $\mathbf F(x) \in p^kM$ implies $\mathbf F(y) \in p^{k-1} M$. By induction on $k$, one has $y \in p^{k-2}M$, whence $x \in p^{k-1}M$. \end{proof}
\begin{cor} Let $M=H^1_{\crys}(A)$. Suppose $j ,k \in \mathbf Z_{\geq 2}$, $x \in M$ and $\mathbf F^j(x) \in p^kM$. Then $\mathbf F^{j-1}(x) \in p^{k-1}M$. \label{cor h1 crys}\end{cor}
\begin{proof} Write $\mathbf F^j(x)=\mathbf F(\mathbf F^{j-1}(x))$ and apply Lemma~\ref{lemma h1 crys}.
\end{proof}
\begin{cor} Let $M=H^1_{\crys}(A)$. Suppose $k \in \mathbf Z_{\geq 1}$, $x \in M$ and $\mathbf F^k(x) \in p^kM$. Then $\bar x \in \fil^1\bar M$. \label{cor h1 crys 2}\end{cor}
\begin{proof} Applying \Cor~\ref{cor h1 crys} repeatedly $k-1$ times  gives $\mathbf F(x) \in pM$. Then the conclusion follows from Mazur's theorem. \end{proof}
\begin{lem} The following inclusion holds: \begin{equation} \fil^1_{\tau_i}\supset \overline{ M_{\tau_i}^{[e_{\tau} ]}} \end{equation} \label{lemma base fil1 slopes}\end{lem}
\begin{proof} Since $M_{\tau_i}^{[e_{\tau}]}$ is isoclinic of slope $e_{\tau}$, we have $\mathbf F^{e_{\tau}}(M_{\tau_i}^{[e_{\tau}]})\subset p^{e_{\tau}}M_{\tau_i}^{[e_{\tau}]}$. Suppose $x \in M_{\tau_i}^{[e_{\tau}]}$. Then $\mathbf F^{e_{\tau}}(x) \in p^{e_{\tau}}M_{\tau_i}^{[e_{\tau}]}$, so the conclusion follows from \Cor~\ref{cor h1 crys 2}.
\end{proof}
\begin{lem} Let $\nu \in\{0,1,\ldots, e_{\tau}-1\}$. Then for all $i \leq e_{\tau}-\nu$, one has \begin{equation}  \fil^1_{\tau_i} \supset M_{\tau_i}^{[e_{\tau}-\nu]}  \label{eq induction fil1 slope decomp}\end{equation} \label{lemma induction fil1 slopes}\end{lem}
\begin{proof} The proof is by induction on $\nu$. The case $\nu=0$ is Lemma~\ref{lemma base fil1 slopes}. Suppose~\eqref{eq induction fil1 slope decomp} holds up to $\nu-1$. Then we have \begin{equation} \fil^1_{\tau_{e_{\tau}-\beta}}=\bigoplus_{j \geq e_{\tau}-\beta}\overline{ M_{\tau_{e_{\tau}-\beta}}^{[j]}} \end{equation} for all $ \beta \leq \nu-1$.

Consider the diagram \begin{equation} M_{\tau_i} \overset{\mathbf F}{\longrightarrow} M_{\sigma\tau_i}\overset{\mathbf F}{\longrightarrow} \cdots \overset{\mathbf F}{\longrightarrow} M_{\sigma^{e_{\tau}-1}\tau_i}\overset{\mathbf F}{\longrightarrow} M_{\tau_i}  \end{equation}
Let $t_1 \geq t_2 \geq \cdots \geq t_{\nu}$ such that for all $\alpha$, $1 \leq \alpha \leq \nu$, one has $\sigma^{t_{\alpha}}\tau_i=\tau_{j_\alpha}$ and $j_{\alpha}>e_{\tau}-\nu$. Let $x \in M_{\tau_i}^{e_{\tau}-\nu}$. Then $\mathbf F^{e_{\tau}}(x) \in p^{e_{\tau}-\nu}M_{\tau_i}$.
By \Cor~\ref{cor h1 crys 2}, we can subtract $e_{\tau} - t - 1$ from the exponents on both sides, thus obtaining:
\begin{equation} \mathbf F^{t_1+1}(x) \in p^{t_1+1-\nu} M_{\sigma^{t_1+1} \tau_i}. \end{equation}

Writing \begin{equation} {\mathbf F} \left( \frac{\mathbf F^{t_1}}{p^{t_1 - \nu}} (x)  \right) \in p M_{\sigma^{t_1+1}\tau_i}, \end{equation} we see by Mazur's theorem that
\begin{equation}     \frac{\mathbf F^{t_1}}{p^{t_1 - \nu}} (\overline{x})  \in \fil^1_{\sigma^{t_1} \tau_i} = \fil^1_{\tau_{j_1}} . \end{equation}

Since $j_1 > e_{\tau} - \nu$, by the induction hypothesis and equality of dimensions, we have
\begin{equation}
\fil^1_{\tau_{j_1}}=\bigoplus_{j \geq j_1}\overline{ M_{\tau_{j_1}}^{[j]}}.
\end{equation}
On the other hand, by assumption, $x \in M_{\tau_i}^{[e_{\tau} - \nu]}$. Since $\mathbf F^{t_1}/p^{t_1-\nu}$ is an isogeny, Lemma~\ref{lemma isogeny preserves slopes} implies that \begin{equation} \frac{\mathbf F^{t_1}}{p^{t_1-\nu}}(x) \in M_{\tau_{j_1}}^{e_{\tau} - \nu} .\end{equation} Hence \begin{equation}  \frac{\mathbf F^{t_1}}{p^{t_1-\nu}}(\bar x) \in \overline{M_{\tau_{j_1}}^{e_{\tau} - \nu}} \cap \bigoplus_{j \geq j_1}\overline{ M_{\tau_{j_1}}^{[j]}}=\{0\} .\end{equation} Therefore $\mathbf F^{t_1}(x) \in p^{t_1-\nu+1}M_{\tau_{j_1}}$.

Repeating the same argument with $t_2$ we obtain $\mathbf F^{t_2}(x) \in p^{t_2-\nu+2}M_{\tau_{j_2}}$. Continuing in this way we finally arrive at $\mathbf F^{t_{\nu}}(x) \in p^{t_{\nu}}M_{\tau_{j_{\nu}}}$ and one last application of \Cor~\ref{cor h1 crys 2} yields $\bar{x} \in \fil^1_{\tau_i}$.

\end{proof}
Lemma~\ref{lemma induction fil1 slopes} completes the proof of \Th~\ref{th fil1 slope decomp}
\end{proof}

\section{The generalized Hasse invariants} \label{sec generalized hasse invariants}

Based on the results and notation of the previous section, we are in position to define the desired generalized Hasse invariants.

Let $\mathcal A \rightarrow sh$ be a representative of the universal isogeny class. The absolute Frobenius morphism \[\mathbf F: \mathcal  A \rightarrow \mathcal A  \] induces a $\sigma$-linear map \begin{equation} \mathbf F: H^1_{\crys}(\mathcal A)  \rightarrow H^1_{\crys}(\mathcal A). \end{equation}

As we have seen in \S\ref{sec f crystals hodge filtration}, this map permutes non-trivially the factors indexed by the embeddings $\tau$. This permutation can be decomposed into cycles according to the orbits $\frak o_{\tau}$. Consider such an orbit $\frak o_{\tau} = \{ \tau_1, \dots, \tau_{e_{\tau}} \}$. Let $ \Grad ^0_{\tau_i}= H^1_{dR}(\calA)_{\tau_i}/\fil^1_{\tau_i}$. Set $d_i=\dim \Grad ^0_{\tau_i}$ and $c_i=(i-1)d_i-(d_1+\cdots +d_{i-1})$.
\begin{lem} The map \begin{equation}\bigwedge^{d_i} \mathbf F^{e_{\tau}}:H^{d_i}_{\crys}(\mathcal A)_{\tau_i} \rightarrow H^{d_i}_{\crys}(\mathcal A)_{\tau_i} \end{equation} is divisible by $p^{c_i}$. \label{lemma hd crys divisibility} \end{lem}
\begin{proof} Since the $\mu$-ordinary locus $sh^{\mu-\ord}$ is open and dense \cite[\Th 1.6.2]{Wed}, it suffices to prove the divisibility for every $\mu$-ordinary geometric point $A$. (We thank David Geraghty for pointing out to us that this follows from \cite{de-jong-crystalline}, specifically remarks in \S\S1.1-1.2 and \S2.3.4 of \loccit, using the fact that our Shimura variety $sh$ is smooth over a field.) By Lemma~\ref{lemma hodge newton polygon} we know that the Hodge polygon of $M_{\tau_i}$ is $\ord_{\frak o_{\tau}}(n, \frak f)$.  Since the smallest slope of the Hodge polygon of $\bigwedge^{d_i}M_{\tau_i}$ is the sum of the $d_i$ smallest slopes of the Hodge polygon of $M_{\tau_i}$, the smallest slope of the Hodge polygon of $\bigwedge^{d_i}M_{\tau_i}$ is \begin{equation} \sum_{j=1}^{i-1} j(d_{j+1}-d_{j})=c_i, \end{equation} so the lemma follows from \cite[1.2.1]{KatzSlopeFiltration}.  \end{proof}

\begin{lem} \label{restrictionToFil} The restriction of the map \begin{equation} \frac{\bigwedge^{d_i}\mathbf F^{e_{\tau}}}{p^{c_i}}:H^{d_i}_{dR}(\mathcal A)_{\tau_i} \rightarrow H^{d_i}_{dR}(\mathcal A)_{\tau_i} \end{equation} to $\fil^1H^{d_i}_{dR}(\mathcal A)_{\tau_i}$ is zero. \label{lemma fil1 zero} \end{lem}
\begin{proof} Again, because the $\mu$-ordinary locus is open and dense \cite[\Th 1.6.2]{Wed},  it suffices to prove the vanishing for every $\mu$-ordinary geometric point $A$. Let \begin{equation}W_{\tau_i}=\bigoplus_{j<i}M_{\tau_i}^{[j]} \end{equation} By \Th~\ref{th fil1 slope decomp}, we have a decomposition \begin{equation} \overline{ M_{\tau_i}}=\fil^1_{\tau_i} \oplus \overline{W_{\tau_i}} \end{equation}Thus \begin{equation} \fil^1H^{d_i}_{dR}(A)_{\tau_i}= \bigoplus_{s=1}^{d_i} \left(\bigwedge^s\fil^1_{\tau_i} \otimes \bigwedge^{d_i-s} \overline{W_{\tau_i}}\right)\end{equation}
Therefore Lemma~\ref{lemma fil1 zero} is equivalent to showing that the restriction of $(\bigwedge^{d_i} \mathbf F^{e_{\tau}})/p^{c_i})$ to $\bigwedge^s\fil^1_{\tau_i} \otimes \bigwedge^{d_i-s}\overline{W_{\tau_i}}$ is zero for all $s \geq 1$. So fix $s$ and let $x\in \bigwedge^s\fil^1_{\tau_i} \otimes \bigwedge^{d_i-s}\overline{W_{\tau_i}}$.

Let \begin{equation} M' = \bigwedge^s\left(\bigoplus_{j \geq i}M_{\tau_i}^{[j]}\right) \otimes \bigwedge^{d_i-s} \left(\bigoplus_{j <i}M_{\tau_i}^{[j]}\right).\end{equation}
By \Th~\ref{th fil1 slope decomp}, there exists a lift $\tilde x$ of $x$ to $H^{d_i}_{\crys}(A)_{\tau_i}$ which lies in $M'$.
The smallest slope $\lambda_{\rm min}$ of the crystal $(M',  (\bigwedge^{d_i} \mathbf F^{e_{\tau}}))$ is, by definition, the sum of the $s$ smallest slopes of $\bigoplus_{j \geq i}M_{\tau_i}^{[j]}$ plus the sum of the $d_i - s $ smallest slopes of $\bigoplus_{j <i}M_{\tau_i}^{[j]}$. Since $s \geq 1$, $\lambda_{\rm min}$ is strictly bigger than the sum of the $d_i$ smallest slopes of $M_{\tau_i}$, and the latter is precisely $c_i$ by definition. Thus, by \cite[1.2.1]{KatzSlopeFiltration}, $(\bigwedge^{d_i} \mathbf F^{e_{\tau}})(\tilde{x}) \in p^{\lambda_{\rm min}} M_{\tau_i}$ and therefore $(\bigwedge^{d_i} \mathbf F^{e_{\tau}}/p^{c_i})(\tilde{x}) \in p M_{\tau_i}$.

\end{proof}

By Lemma~\ref{lemma fil1 zero}, we get an induced map \begin{equation} \frac{\bigwedge^{d_i}F^{e_{\tau}}}{p^{c_i}}:\Grad^0 H^{d_i}_{dR}(\mathcal A)_{\tau_i} \rightarrow \Grad^0 H^{d_i}_{dR}(\mathcal A)_{\tau_i} \end{equation}
Since $\Grad^0 H^{d_i}_{dR}(\mathcal A)_{\tau_i}\cong \omega_{\tau_i}^{\vee}$, we obtain a section \begin{equation}^{\tau_i} \! H \in H^0(sh, \omega_{\tau_i}^{p^{e_{\tau}}-1}). \end{equation}

\begin{dfn}
The section $^{\tau_i}\!H$  is called the $\tau_i$-Hasse invariant of $sh$.
\end{dfn}
\noindent
As in \cite[\Th 4.2.1]{WushiGalrepshldsI}, the $\tau_i$-Hasse invariant is compatible with isogenies in the sense that if $\varphi: \mathcal A \rightarrow \mathcal B$ is an isogeny preserving the $\mathcal U\ul$-structure then $\varphi^*(^{\tau_i}\!H(\mathcal B))=\ \! ^{\tau_i}\!H(\mathcal A)$. Therefore the $\tau_i$-Hasse invariant is well-defined.

\begin{remark} Compatibility with {\rm \cite{GoldringNicole}: If $F^+ = \mathbf Q$, then $^{\tau_1}\!H$ is equal to the $\mu$-Hasse invariant of \loccit \hspace{.01in} (see Appendix B).}
\end{remark}

We are now in a position to define the $\mu$-ordinary Hasse invariant in complete generality for unitary Shimura varieties.

\begin{dfn} Let $m = \underset{\tau \in \mathcal T}{\lcm} \{p^{e_\tau}-1 \}$, and let $m_\tau = m / (p^{e_{\tau}}-1)$. We define the $\mu$-ordinary Hasse invariant $\Ha$ as the product:
\begin{equation}
\Ha = \prod_{\tau \in \mathcal T} \ (^{\tau}\!H)^{m_\tau} \in H^0\big(sh, \omega^{m} \big)
\end{equation}
 \label{def exponent m} \end{dfn}

\section{The non-vanishing loci of the Hasse invariants}

We will describe the non-vanishing locus of the $\tau$-Hasse invariant one embedding $\tau$ at a time. In the end, the non-vanishing locus of the $\mu$-ordinary Hasse invariant will easily be read off as the $\mu$-ordinary locus.

\begin{thm}
Let $A$ be a geometric point of the special fiber $sh_{K^{(p)},\gerp}$. Then \[ ^{\tau_i}H (A) \neq 0 \] if and only if the Newton polygon of $M_{\tau_i}$ meets $\ord_{\frak o_{\tau_i}}(n,\frak f)$ at $(d_i,c_i)$ in the notation of \S\ref{sec generalized hasse invariants}.

%the $i$-th breakpoint of $\ord_{\frak o_{\tau_i}}(n,\frak f)$.
\label{th vanishing tau ha}\end{thm}

\begin{proof}
By Rapoport-Richartz's version of Mazur's inequality (see \cite[Lem. 1.3.4]{MoonenSerreTate}), the Newton polygon of $M_{\tau_i}$ sits on or above the ordinary polygon $\ord_{\frak o_{\tau_i}}(n,\frak f)$. Let $(d_i,g_i)$ be the unique point on the Newton polygon of $M_{\tau_i}$ whose first coordinate is $d_i$. Since the point $(d_i,c_i)$ lies on the polygon $\ord_{\frak o_{\tau_i}}(n,\frak f)$, the point $(d_i, g_i)$ is on or above $(d_i,c_i)$, meaning that  $g_i \geq c_i$. The rational number $g_i$ is the sum of the first $d_i$ slopes of the Newton polygon of $M_{\tau_i}$, hence $g_i$ is the smallest slope of the Newton polygon of $\bigwedge^{d_i} M_{\tau_i}$. Therefore the smallest Newton slope of \[ \left(\bigwedge^{d_i} M_{\tau_i}, \frac{\bigwedge^{d_i} \mathbf F ^{e_{\tau}}}{p^{c_i}} \right) \] is $g_i - c_i$. By \cite[1.3.3]{KatzSlopeFiltration} the action of $\bigwedge^{d_i} \mathbf F ^{e_{\tau}}/p^{c_i}$ on $\overline{\bigwedge^{d_i}M_{\tau_i}}$ is nilpotent if and only if this smallest slope is positive i.e. if and only if $g_i>c_i$.

Since $\Grad^0 H^{d_i}_{dR}(A)_{\tau_i}$ is a line,  $\bigwedge^{d_i} \mathbf F ^{e_{\tau}}/p^{c_i}$ acts on it by a scalar, namely $^{\tau_i}\!H(A)$. By Lemma~\ref{lemma fil1 zero},  the action of $\bigwedge^{d_i} \mathbf F ^{e_{\tau}}/p^{c_i}$ on $\overline{\bigwedge^{d_i}M_{\tau_i}}$ is nilpotent if and only if $^{\tau_i}\!H(A)=0$.
\end{proof}
\begin{cor} Let $A$ be a geometric point of the special fiber $sh_{K^{(p)},\gerp}$. Then $\Ha (A) \neq 0$ if and only if $A$ is $\mu$-ordinary. \label{cor vanishing mu ha}\end{cor}
\begin{proof} By definition $\Ha (A) \neq 0$ if and only if $^{\tau}\! H(A)\neq 0$ for all $\tau \in \mathcal T$. By \Th~\ref{th vanishing tau ha}, for every orbit $\frak o_{\tau}$ we have that $^{\tau'}\! H(A) \neq 0$ for all $\tau' \in \frak o_{\tau}$ if and only if the Newton polygon of $M_{\tau}$ meets $\ord_{\frak o_{\tau}}(n,\frak f)$ at every breakpoint of $\ord_{\frak o_{\tau}}(n, \frak f)$, so  $^{\tau'}\! H(A) \neq 0$ for all $\tau' \in \frak o_{\tau}$ if and only if the Newton polygon of $M_{\tau}$ equals $\ord_{\frak o_{\tau}}(n, \frak f)$. An application of Lemma~\ref{lemma orbits mu ordinary} completes the proof.
\end{proof}

\begin{proof}[Proof of {\rm \Th~\ref{th mu ordinary hasse invariant}}] \Cor~\ref{cor vanishing mu ha} establishes ($\mu$-Ha1). Properties ($\mu$-Ha2)-($\mu$-Ha4) are proved in exactly the same way as in \cite[Lemma 4.4.1, \Th 6.2.1]{WushiGalrepshldsI} \end{proof}

\section{Correction to \cite{WushiGalrepshldsI}} \label{sec correction}
Thanks to Jay Pottharst for pointing out the need to make the following minor modifications in \cite[\S6.2]{WushiGalrepshldsI}:  In the second and third sentences of the proof of \Th~6.2.1, the phrase ``is non-zero'' (resp. ``is also non-zero'') should be replaced with the phrase ``is a non-zero divisor'' (resp. ``is also a non-zero divisor''). Moreover, in the third sentence, the word ``separable'' should be replaced with the word ``finite''. Finally, in the fifth sentence, the phrase ``Since the product of two sections that are each non-zero modulo $\lambda$'' should be replaced with the phrase ``Since the product of a section which is non-zero modulo $\lambda$ with section which is a non-zero divisor modulo $\lambda$''.
\section*{Acknowledgments} \label{sec acknowledgements}

We thank G. Faltings for suggesting the idea of using crystalline cohomology to the second author. We are grateful to P. Deligne for his correspondence with the first author showing how to implement this idea in a concrete example.

\appendix
\section{The point of view of Ekedahl-Oort} \label{sec eo}
We keep the notation introduced in the main text. In particular, recall that $\mathfrak o_{\tau}$ denotes the orbit of embeddings of $\tau$ under Frobenius and $e_{\tau}$ denotes the cardinality of this orbit.

We begin by recalling Moonen's definition of ``standard ordinary objects'' \cite[1.2.3]{MoonenSerreTate}. Given an orbit $\frak o_{\tau}$ and its type $(n, \frak f)$, we have a Dieudonn\'e module $M^{{\ord}_{\frak o_{\tau}}}(n, \frak f)$ defined as follows: As $W(k)$-module, let $M^{{\ord}_{\frak o_{\tau}}}(n, \frak f)$ be the free module generated by the basis consisting of symbols $\epsilon_{\tau_i, j}$ such that $\tau_i \in \frak o_{\tau}$ and $1 \leq j \leq n$. On this basis, Frobenius acts by
\begin{equation} F(\epsilon_{\tau_i,j})=\left\{ \begin{array}{ccc} \epsilon_{\sigma \tau_i,j} & \mbox{ if } & \frak f(\tau_i) \leq n-j \\ p \epsilon_{\sigma \tau_i,j} & \mbox{ if } & \frak f(\tau_i) > n-j\end{array} \right. \label{eq def std ord obj 1} \end{equation}
and Verschiebung is given by
\begin{equation} V(\epsilon_{\sigma\tau_i,j})=\left\{ \begin{array}{ccc} p \epsilon_{ \tau_i,j} & \mbox{ if } & \frak f(\tau_i) \leq n-j \\ \epsilon_{\tau_i,j} & \mbox{ if } & \frak f(\tau_i) > n-j\end{array} \right. \label{eq def std ord obj 2} \end{equation}
Put $M_{\tau_i}^{{\ord}_{\frak o_{\tau_i}}}(n, \frak f)=\mbox{span}(\{\epsilon_{\tau_i, j}|1\leq j \leq n\})$. Note that the module $M_{\tau_i}^{{\ord}_{\frak o_{\tau_i}}}(n, \frak f)$ is stable under $F^{e_{\tau_i}}$.

The key role played by the modules $M^{{\ord}_{\frak o_{\tau}}}(n, \frak f)$ and $M_{\tau_i}^{{\ord}_{\frak o_{\tau_i}}}(n, \frak f)$ stems from the following result of Moonen:
\begin{thm}[Moonen \cite{MoonenSerreTate}, \Th 1.3.7] Let $A$ be a geometric point of the special fiber $sh$. Then $A$ is $\mu$-ordinary if and only if the Dieudonn\'e module of $A$ is isomorphic to \begin{equation} \bigoplus_{{\rm orbits } \ \frak o_{\tau}}M^{{\ord}_{\frak o_{\tau}}}(n, \frak f)^{\oplus r}. \label{eq moonen my p and final ordinary}\end{equation} \label{th moonen mu p and final ordinary} \end{thm}
Henceforth assume $A$ is a $\mu$-ordinary geometric point of the special fiber $sh_{K^{(p)},\gerp}$. We identify the Dieudonn\'e module~\eqref{eq moonen my p and final ordinary} with $H^1_{\crys}(A)$ in such a way that the Frobenii $F$ and $\mathbf F$ correspond to one another. The submodule $M_{\tau_i}^{{\ord}_{\frak o_{\tau_i}}}(n, \frak f)$ then corresponds to $H^1_{\crys}(A)_{\tau_i}$.

In the basis $\mathcal B_{d_i}=\{\epsilon_{\tau_i,j}| 1\leq j \leq n\}$, the matrix of $\mathbf F^{e_{\tau_i}}$ acting on $M_{\tau_i}^{{\ord}_{\frak o_{\tau_i}}}(n, \frak f)$ is the diagonal matrix $\diag(p^{a_1}, \ldots ,p^{a_n})$, where the $a_j$ are the slopes of $\ord_{\frak o_{\tau_i}}(n, \frak f)$, whose definition was recalled in~\eqref{eq def slope ord}. Therefore the matrix of $\wedge^{d_i}\mathbf F^{|\frak o_{\tau_i}|}$ in the basis $\mathcal B_{d_i}=\{\epsilon_{\tau_i,j_1}\wedge \cdots \wedge \epsilon_{\tau_i,j_{d_i}} | 1 \leq j_1 < \cdots <j_{d_i}\leq n\}$ is the diagonal matrix with entry \begin{equation} p^{a_{j_1}+\cdots+a_{j_{d_i}}} \end{equation} corresponding to the basis vector $\epsilon_{\tau_i,j_1}\wedge \cdots \wedge \epsilon_{\tau_i,j_{d_i}}$. Since $c_i$ is the sum of the $d_i$ smallest slopes of $M_{\tau_i}^{{\ord}_{\frak o_{\tau_i}}}(n, \frak f)$, we see that $\wedge^{d_i}\mathbf F^{e_{\tau_i}}$ is divisible by $p^{c_i}$, thus reproving Lemma~\ref{lemma hd crys divisibility}.

Applying Mazur's theorem (\Th~\ref{th frobenius determines the hodge filtration}) to~\eqref{eq def std ord obj 1}, we see that \begin{equation} \fil^1_{\tau_i}=\overline{\mbox{span}(\{\epsilon_{\tau_i,j}|\frak f(\tau_i)>n-j\})} \end{equation}

Hence \begin{equation} \fil^1H^{d_i}_{dR}(A)_{\tau_i}=\overline{\mbox{span}(\mathcal B_{d_i}-\{\epsilon_{\tau_i,1}\wedge \cdots \wedge \epsilon_{\tau_i,{d_i}}\})} \label{eq fil1 std ord obj}\end{equation}
It follows from the description of the matrix of $\wedge^{d_i}\mathbf F^{|\frak o_{\tau_i}|}$ in the basis $\mathcal B_{d_i}$ that $p^{c_i+1}$ divides $\wedge^{d_i}\mathbf F^{e_{\tau_i}}(\mbox{span}(\mathcal B_{d_i}-\{\epsilon_{\tau_i,1}\wedge \cdots \wedge \epsilon_{\tau_i,{d_i}}\}))$. Combining this with~\eqref{eq fil1 std ord obj} reproves Lemma~\ref{lemma fil1 zero}. We also get that $\wedge^{d_i}\mathbf F^{ e_{\tau_i}}/p^{c_i}$ is non-zero on $\Grad^0 H^{d_i}_{dR}(A)_{\tau_i}$, from which we recover the ``if'' part of \Cor~\ref{cor vanishing mu ha}.
% For the ``only if" part, it seems one has to use Newton polygons as was done in the main body of the text.

\section{An elementary construction for the case $F^+ = \mathbf Q$} \label{sec F^+=Q}

Suppose that $F^+ = \mathbf Q$, and therefore $\mathbf G (\mathbf R) = \mathbf{GU}(a,b)$ for some positive integers $a, b$. Assume henceforth, without loss of generality, that $a \leq b$. The assumption of \S\ref{sec main result} that $p$ is a prime of good reduction for $\mathcal U$ implies that $p$ is unramified in $E$. If $a=b$ then $E=\mathbf Q$, so $p$ is necessarily split in $E$ and the classical ordinary locus is open dense. Hence we assume from now on that $a<b$ and that $p$ is inert in $E$. It follows that the Hodge bundle $\Omega$ decomposes over $E$ as \begin{equation} \Omega=\Omega_{a}^{\oplus r} \oplus \Omega_{b}^{\oplus r} \label{eq hodge bundle decomposition}, \end{equation} where $\Omega_{a}$ (resp. $\Omega_{b}$) has rank $a$ (resp. $b$) and $r$ is the rank of $B$ over $F$. Let $\omega_{a}$ (resp. $\omega_{b}$) be the determinant of $\Omega_{a}$ (resp. $\Omega_{b}$).

Let $\mathcal A$ be an abelian scheme representing the universal isogeny class above $sh$.
The Verschiebung ${\rm Ver:} \ \mathcal A\ul \rightarrow A$ induces a map \begin{equation} {\rm Ver}^*: \Omega  \rightarrow \Omega\ul \label{eq ver hodge bundle}. \end{equation} Since $\frak p$ is inert, the restrictions of ${\rm Ver}^*$ to $\Omega_{a}$ (resp. $\Omega_{\mathcal K\ul,b}$) have the form \begin{equation} {\rm Ver}^*_{| \Omega_{a}}:\Omega_{a} \longrightarrow \Omega_{b}^{(p)} \mbox{ and } {\rm Ver}^*_{| \Omega_{b}}:\Omega_{b} \longrightarrow \Omega_{a}^{(p)} \label{eq ver permutes a b} . \end{equation}

Therefore, if $({\rm Ver}^*)^2$ denotes the composite of ${\rm Ver}^*$ with itself, then we have \begin{equation} ({\rm Ver}^*)^2_{| \Omega_{a}}: \Omega_{a} \longrightarrow \Omega_{a}\uls \label{eq ver2}.  \end{equation}

% and iterating $a+b$ times gives
% \begin{equation}({\rm Ver}^*)^{2(a+b)}_{| \Omega_{\mathcal K\ul, a}}: \Omega_{\mathcal K, a} \longrightarrow \Omega_{\mathcal K, a}^{(\ell^{2(a+b)})} \label{eq ver2(a+b)}.  \end{equation}

Let \begin{equation}  ^{\mu}\!h(\mathcal A): \omega_{ a} \longrightarrow \omega_{ a} ^{ p^{2}} \label{eq mu hasse map}, \end{equation} be the top exterior power of that map, where we have used that $\omega_{ a}^{(p^{2})}= \omega_{ a} ^{p^{2}}$ since $\omega_{a}$ is a line bundle. The map $^{\mu}\!h(\mathcal A)$ induces a global section \begin{equation} ^{\mu}\!H(\mathcal A) \in H^0(sh, \omega_{a}^{p^2-1}) \label{eq mu hasse invariant}. \end{equation}

If $\mathcal B$ is another representative of the universal isogeny class above $sh$ and $\varphi: \mathcal A \rightarrow \mathcal B$ is an isogeny compatible with the endomorphism actions of $\mathcal A, \mathcal B$, then as in \cite[\S4.2]{WushiGalrepshldsI}, the compatibility of Verschiebung with isogenies (Lemma 4.2.3 of \loccit) implies that $\varphi^*(^\mu\!H(\mathcal B))= {^\mu\!H}(\mathcal A)$. Hence we may omit reference to the representatives $\mathcal A$ or $\mathcal B$ and we have a section $^\mu \!H \in H^0(sh, \omega_{a}^{p^{2}-1}$.

\begin{rmk}
Note that applying the above construction is entirely done modulo $p$. Applying it to $\omega_{b}$ gives nothing but the zero section. With hindsight, this shows the necessity to lift our setup to characteristic zero to divide by higher powers of $p$.
\end{rmk}

\begin{lem} \label{lemma simpleNP}
 The Newton polygon $\mathcal N^{\rm ord}$ of the underlying isogeny class of abelian schemes of a $\mu$-ordinary geometric point of $sh$ has the following slopes:
 \[
\left. \begin{array}{l} 0 \\ 1/2 \\ 1 \end{array}
\right\} \text{with multiplicity } \left\{ \begin{array}{r} 2ar \\ 2(b-a)r \\ 2ar \end{array} \right.
\]

\end{lem}

\begin{proof} The case $r=1$ follows from \cite[2.3.2]{Wed}. The case of general $r$ follows subsequently from \cite[1.3.1 and 3.2.9]{MoonenSerreTate}.
\end{proof}
\begin{prop} The $\mu$-ordinary locus is the maximal $p$-rank stratum of $sh$. \label{cor mu ord max p rank}
\end{prop}
\begin{proof} The key point is that, by \cite[Prop. 2.4(iv) and \Th 4.2]{RapRic}, the Newton polygon $\mathcal N^{\rm ord}$ described in Lemma~\ref{lemma simpleNP} is the lowest among the Newton polygons of the underlying isogeny classes of abelian schemes corresponding to geometric points of $sh$. Let $A$ be an abelian scheme with Newton polygon $\mathcal N(A)$. Then $\mathcal N(A)$ is symmetric and the $p$-rank of $A$ is the multiplicity of 0 (=the multiplicity of 1) as a slope of $\mathcal N(A)$.  But if the multiplicity of 0 in $\mathcal N(A)$ is at least the multiplicity of 0 in $\mathcal N^{\rm ord}$ and $\mathcal N(A)$ lies on or above $\mathcal N^{\rm ord}$, then by Lemma~\ref{lemma simpleNP} we must have $\mathcal N(A)=\mathcal N^{\ord}$.
\end{proof}
\begin{cor} The maximal $p$-rank stratum of $sh$ has $p$-rank $2ar$. \label{cor max p rank 2ar} \end{cor}
\begin{proof} This follows directly from Lemma~\ref{lemma simpleNP} and the proof of Prop. ~\ref{cor mu ord max p rank}.
\end{proof}

\begin{lem} Suppose $A$ is an abelian scheme which is a representative of the underlying isogeny class of a geometric point of $sh$. Then $^\mu\!H(A) \neq 0$ if and only if the $p$-rank of $A$ is equal to $2ar$. \label{lemma hasse p rank}\end{lem}

\begin{proof} One has $H^1(A, \mathcal O_A) \cong H^0(A, \Omega^1_A)$ and under this isomorphism the action of Frobenius on $H^1(A, \mathcal O_A)$ corresponds to that of Verschiebung on $H^0(A, \Omega^1_A)$. Hence \cite[\S15]{MumAbVar} implies that the $p$-rank of $A$ equals the semisimple rank of $({\rm Ver}^*)^j:\Omega \rightarrow \Omega^{(p^j)}$  for all $j \in \mathbf N$. Since $\dim A=(a+b)r$, keeping in mind~\eqref{eq hodge bundle decomposition} and using the last corollary of \S14 of \loccit, $({\rm Ver}^*)^j$ is semisimple for $j \geq a+b$. Therefore the $p$-rank of $A$ equals the rank of $({\rm Ver}^*)^j$ for $j \geq a+b$. We take the $(a+b)$th power of the section $^\mu\!H(A)$, see (\ref{eq mu hasse map}). It is clear that  $^\mu\!H(A) \neq 0$ if and only if  $^\mu\!H(A)^n \neq 0$ for any $n \in \mathbf N_{>0},$ in particular for $n=a+b$.

Since $a \leq b$, both ${\rm Ver}^*_{|\Omega_{a}}$ and ${\rm Ver}^*_{|\Omega_{b}}$ have rank at most $a$. So also $({\rm Ver}^*)^j_{|\Omega_{\mathcal K\ul,a}}$ and $({\rm Ver}^*)^j_{|\Omega_{\mathcal K\ul,b}}$ each have rank at most $a$. By ~\eqref{eq hodge bundle decomposition}, $({\rm Ver^*})^j$ has rank at most $2ar$.

The $p$-rank of $A$ equals $2ar$ if and only if the rank of $({\rm Ver}^*)^j$ is $2ar$ for $j \geq a+b$. In turn, the rank of $({\rm Ver}^*)^j$ is $2ar$ if and only if both ${\rm Ver}^*_{|\Omega_{a}}$ and ${\rm Ver}^*_{|\Omega_{b}}$ have rank $a$. Since $\Omega_{a}$ and $\Omega_{a}^{(p^{2(a+b)})}$ are rank $a$ vector bundles, the determinant of a map between them is nonzero if and only if it has rank $a$.
\end{proof}

The main properties of the Hasse invariant follow by standard arguments. For the liftability to characteristic zero, we may cite \cite[Prop. 7.14]{LanSuhgen} (or \cite{LanSuhcpt} in the compact case), to argue that there exists $k \in \mathbf N$ such that $\omega_{a}^{k(p^{2}-1)}$ itself extends to an ample line bundle on the minimal compactification $sh^{\rm min}$.

\bibliographystyle{amsalpha}
\bibliography{galrepsholodsunitary}
\end{document}